\documentclass[10pt]{article}
\usepackage[margin=1.5in]{geometry}  
\usepackage{amsmath}
\usepackage{amsfonts}
\usepackage{amssymb}
\usepackage{amsthm}
\usepackage{cases}
\usepackage{stmaryrd}
\usepackage{latexsym}
\usepackage{mathrsfs}
\usepackage{graphicx}
\usepackage{verbatim}

\newtheorem{thm}{Theorem}[section]

\newtheorem{lem}[thm]{Lemma}
\newtheorem{prop}[thm]{Proposition}

\newtheorem{rem}[thm]{Remark}

\numberwithin{equation}{section}

\def\C{\mathbb{C}}

\def\N{\mathbb{N}}

\def\Q{\mathbb{Q}}
\def\R{\mathbb{R}}

\def\DD{\mathcal{D}}
\def\EE{\mathcal{E}}
\def\HH{\mathcal{H}}

\def\SS{\mathcal{S}}

\def\LLL{\mathscr{L}}

\def\supp{\text{\rm supp}}

\def\range{\rm ran}
\def\sgn{\rm sgn}

\def\lan{\langle}
\def\ran{\rangle}
\def\da{\downarrow}
\def\ra{\rightarrow}
\def\bs{\backslash}
\def\ol{\overline}

\def\ep{\epsilon}

\def\la{\lambda}

\def\si{\sigma}
\def\Si{\Sigma}
\def\om{\omega}
\def\Om{\Omega}
\def\de{\delta}
\def\De{\Delta}
\def\ga{\gamma}

\def\vp{\varphi}

\begin{document}

\nocite{*}

\title{Completeness for Sparse Potential Scattering}

\author{Zhongwei Shen\footnote{Email: zzs0004@auburn.edu}\\Department of Mathematics and Statistics\\Auburn University\\Auburn, AL 36849\\USA}

\date{}

\maketitle

\begin{abstract}
The current paper is devoted to the scattering theory of a class of continuum Schr\"{o}dinger operators with deterministic sparse potentials. We first establish the limiting absorption principle for both modified free resolvents and modified perturbed resolvents. This actually is a weak form of the classical limiting absorption principle.  We then prove the existence and completeness of local wave operators, which, in particular, imply the existence of wave operators. Under additional assumptions on the sparse potential, we prove the completeness of wave operators. In the context of continuum Schr\"{o}dinger operators with sparse potentials, this paper gives the first proof of the completeness of wave operators.
\end{abstract}


\section{Introduction}\label{sec-intro}

This paper is concerned with the scattering theory (especially the completeness of wave operators) of a class of Schr\"{o}dinger operators with sparse potentials, that is,
\begin{equation}\label{main-model}
H=H_{0}+V\quad\text{on}\quad L^{2}(\R^{d}),
\end{equation}
where $H_{0}=-\De$ is the negative Laplacian and $V$ is the sparse potential of the form $\sum_{n=1}^{\infty}v_{n}(\cdot-x_{n})$,
in which, $\{v_{n}\}_{n\in\N}$ are real-valued, uniformly bounded and  uniformly compactly supported functions on $\R^{d}$, and $\{x_{n}\}_{n\in\N}\subset\R^{d}$ is the sequence satisfying certain sparse condition. 

As an important ingredient of the spectral theory of Schr\"{o}dinger operators, both discrete and continuum Schr\"{o}dinger operators with sparse potentials have attracted a lot of attention (see \cite{De08,HK00,JL03,JP09,Ki02,Kri93,Kru04,KLS98,Mo98,MV99,MV00,Pe79,Po08} and references therein). In particular, in one dimension, if $v_{n}=a_{n}v$ with $v$ being nonzero, nonnegative, continuous and compactly supported and $\{a_{n}\}_{n\in\N}$ satisfying $a_{n}\ra0$ as $n\ra\infty$, and the sparse sequence $\{x_{n}\}_{n\in\N}$ satisfies $\frac{x_{n}}{x_{n+1}}\ra0$ as $n\ra\infty$, then the spectrum is purely absolutely continuous (a.c.) on $\R_{+}(=[0,\infty))$ if $\{a_{n}\}_{n\in\N}\in\ell^{2}(\N)$, and is purely singularly continuous otherwise (see \cite{KLS98}). In higher dimensions, $d\geq3$, the situation is a little different. Hundertmark and Kirsch proved in \cite{HK00} that the a.c. spectrum contains $\R_{+}$ if $\inf_{m\neq n}|x_{m}-x_{n}|\gtrsim n^{\ga}$ for some $\ga>\frac{2}{d(d-2)}$. In \cite{Kru04}, Krutikov obtained similar results under slightly different assumptions. 

In \cite{HK00,Kru04}, the a.c. spectrum containing $\R^{+}$ is established as a result of the existence of wave operators. Their methods are Cook's method in nature. Since Klaus's theorem (see e.g. \cite{CFKS87,HK00,Kla83}) says that Schr\"{o}dinger operators with sparse potentials may have essential spectrum below zero, we can not simply exclude the possible existence of a.c. spectrum below zero, and therefore, we don't expect the completeness of wave operators in general. But, we can still expect the completeness of local wave operators. This is one goal of the current paper. We also know from Klaus's theorem the essential spectrum below zero, if exists, is characterized by the discrete spectrum of $H_{0}+v_{n}$, $n\in\N$. Therefore, under additional assumptions on $\{v_{n}\}_{n\in\N}$, the completeness may be obtained. This is the other goal of the paper.

Consider the Schr\"{o}dinger operator $H$ in \eqref{main-model}. We attempt to prove the existence and completeness of (local) wave operators by means of the smooth method of Kato (see e.g. \cite{Ka65,RS78} and Appendix \ref{smooth-perturb}). By writing
\begin{equation*}
H-H_{0}=|V|^{1/2}V^{1/2},
\end{equation*}
where $V^{1/2}={\sgn}(V)|V|^{1/2}$ so that $V=|V|^{1/2}V^{1/2}$, this requires to show the local $H$-smoothness of $|V|^{1/2}$ and local $H_{0}$-smoothness of $V^{1/2}$, which for short-range potentials comes from the classical limiting absorption principle, which, however, is not expected to be true in our case since the sparse potential considered here may not be decaying. Instead, we want to prove some properties about the resolvents that closely relate to Kato's smooth method. After considering this and the resolvent identity, we define for $\la>0$ and $\ep>0$ the modified free resolvent and the modified perturbed resolvent, that is,
\begin{equation*}
\begin{split}
F(\la+i\ep)&=|V|^{1/2}(H_{0}-\la-i\ep)^{-1}V^{1/2},\\
P(\la+i\ep)&=|V|^{1/2}(H-\la-i\ep)^{-1}V^{1/2}.
\end{split}
\end{equation*}
Our first result is summarized in

\begin{thm}\label{thm-main}
Let $d\geq2$ and consider 
\begin{equation*}
H=H_{0}+V=H_{0}+\sum_{n=1}^{\infty}v_{n}(\cdot-x_{n})\quad\text{on}\quad L^{2}(\R^{d}), 
\end{equation*}
where $\{v_{n}\}_{n\in\N}$ are real-valued, uniformly bounded, i.e., $\sup_{n\in\N}\|v_{n}\|_{\infty}<\infty$, and uniformly compactly supported, i.e., there's some bounded set $\Om\subset\R^{d}$ such that
\begin{equation*}
\supp(v_{n})\subset\Om\quad\text{for all}\,\,n\in\N,
\end{equation*}
and $\{x_{n}\}_{n\in\N}\subset\R^{d}$ satisfies the sparse condition
\begin{equation*}
{\rm dist}\big(x_{n},\{x_{m}\}_{m\neq n}\big)\geq Cn^{\ga},\quad n\in\N
\end{equation*}
for some $C>0$ and $\ga>\frac{2}{d-1}$. Then, the following statements hold:
\begin{itemize}
\item[\rm(i)] For any $0<a<b<\infty$, we have
\begin{equation*}
\sup_{\la\in[a,b],\ep\in(0,1]}\|F(\la+i\ep)\|_{\LLL(L^{2}(\R^{d}))}<\infty\quad\text{and}\quad\sup_{\la\in[a,b],\ep\in(0,1]}\|P(\la+i\ep)\|_{\LLL(L^{2}(\R^{d}))}<\infty,
\end{equation*}
where $\LLL(L^{2}(\R^{d}))$ is the space of bounded linear operators on $L^{2}(\R^{d})$.

\item[\rm(ii)] For any $I=[a,b]\subset(0,\infty)$, the local wave operators
\begin{equation*}
\Om_{\pm}(H,H_{0};I)=s\mbox{-}\lim_{t\ra\pm\infty}e^{iHt}e^{-iH_{0}t}\chi_{I}(H_{0})
\end{equation*}
exist and are complete. In particular, the wave operators 
\begin{equation*}
\Om_{\pm}(H,H_{0})=s\mbox{-}\lim_{t\ra\pm\infty}e^{iHt}e^{-iH_{0}t}
\end{equation*}
exist.
\end{itemize}
\end{thm}

Theorem \ref{thm-main}$\rm(i)$ is the analog of the classical limiting absorption principle (see e.g. \cite{Ag75,RS78}), which, for short-range potentials, implies the supremums as in Theorem \ref{thm-main}$\rm(i)$. Therefore, we here actually obtain a weak form. It's easy to see that this weak form closely relates to Kato's smooth method, and therefore, Theorem \ref{thm-main}$\rm(ii)$ is almost a direct consequence of Theorem \ref{thm-main}$\rm(i)$. Note that while the results in \cite{HK00,Kru04} hold for $d\geq3$, our results are also true in the dimension $d=2$, which, for continuum models, seems unknown, even for the existence of wave operators.

Next, we put more conditions on $\{v_{n}\}_{n\in\N}$, which includes the case that $v_{n}$, $n\in\N$ are the same, to ensure the completeness of wave operators. 

\begin{thm}\label{thm-main-1}
Let the assumptions of Theorem \ref{thm-main} be satisfied. Suppose further that $\{v_{n}\}_{n\in\N}$ satisfies one of the following two conditions:
\begin{itemize}
\item[\rm(i)] $\|v_{n}\|_{\infty}\ra0$ as $n\ra\infty$, that is, $V$ is decaying;
\item[\rm(ii)] the set $\EE$ has Lebesgue measure zero, where 
\begin{equation*}
\EE=\big\{E<0\big|\exists n_{k}\,\,\text{and}\,\,E_{n_{k}}\in\si(H_{0}+v_{n_{k}})\,\,\text{s.t.}\,\,E_{n_{k}}\ra E\,\,\text{as}\,\,k\ra\infty\big\}.
\end{equation*}
\end{itemize} 
Then, the wave operators $\Om_{\pm}(H,H_{0})$ exist and are complete. In particular, $\si_{ac}(H)=\R_{+}$.
\end{thm}

The assumption$\rm(ii)$ in the statement of Theorem \ref{thm-main-1} is made according to Klaus's theorem (see e.g. \cite{CFKS87,HK00,Kla83}). We remark that the set $\EE$ can have positive Lebesgue measure (see Remark \ref{rem-final}$\rm(i)$).

As an immediate result of Theorem \ref{thm-main-1}, we obtain

\begin{thm}\label{cor-main}
Let $d\geq2$. Consider the Schr\"{o}dinger operator 
\begin{equation*}
H=H_{0}+\sum_{n=1}^{\infty}v(\cdot-x_{n})\quad\text{on}\quad L^{2}(\R^{d}),
\end{equation*}
where $v:\R^{d}\ra\R$ is bounded and compactly supported, and $\{x_{n}\}_{n\in\N}\subset\R^{d}$ satisfies the sparse condition
\begin{equation*}
{\rm dist}\big(x_{n},\{x_{m}\}_{m\neq n}\big)\geq Cn^{\ga},\quad n\in\N
\end{equation*}
for some $C>0$ and $\ga>\frac{2}{d-1}$. Then, the wave operators $\Om_{\pm}(H,H_{0})$ exist and are complete.
\end{thm}

Here, we want to point out the work of Poulin (see \cite{Po08}) and Jak\v{s}i\'{c} and Poulin (see \cite{JP09}) for discrete Schr\"{o}dinger operators with random sparse potentials. With the presence of randomness, one is only interested in almost sure properties, and therefore, the completeness of wave operators is always expected. This is indeed a big difference between deterministic and random operators. Actually, both the above two papers obtain the existence and completeness of wave operators. In terms of the methods, while both papers are based on the Jak\v{s}i\'{c}-Last criterion of completeness (see \cite{JL03}), which holds only for discrete models,  the former uses probabilistic arguments and the later uses analytic Fredholm theorem and is deterministic in nature. 

There's a random version of of the operator \eqref{main-model}, that is, 
\begin{equation*}
H_{\om}=H_{0}+\sum_{n=1}^{\infty}\om_{n}v_{n}(\cdot-x_{n}),
\end{equation*}
where $\{\om_{n}\}_{n\in\N}$ are random variables. The existence of wave operators has been established in \cite{Kru04}. Due to the present of randomness, we also expect the completeness. Moreover, considering the localization phenomenon, the scenario for $H_{\om}$ shoud be: almost surely, wave operators exsit and are complete, and outside $\R_{+}$, the spectrum is pure point. This has been established for discrete models (see \cite{Po08}), and therefore, it is believed to be true for continuum models. We hope to prove it in our futher work.

The rest of the paper is organized as follows. In Section \ref{sec-settings}, we give some basic settings about the model. Section \ref{sec-free} and \ref{sec-perturbed} are the main parts of the paper. In Section \ref{sec-free}, we establish the limiting absorption principle of $F(z)$. In Section \ref{sec-perturbed}, we first show the invertibility of $1+F(z)$, then prove the uniform boundedness of $(1+F(z))^{-1}$, and finally, establish the limiting absorption principle of $P(z)$. The final section, Section \ref{sec-proof-main-results}, is devoted to the proof of Theorem \ref{thm-main}, Theorem \ref{thm-main-1} and Theorem \ref{cor-main}. We also make some remarks about the resuts of the current paper there.

Throughout the paper, we use the following notations: for two nonnegative numbers $a$ and $b$, by writing $a\lesssim b$ we mean there's some $C>0$ such that $a\leq Cb$; similarly, we also use $a\gtrsim b$; $B_{R}(x)$ denotes the open ball in $\R^{d}$ centered at $x$ with radius $R$; its closure is denoted by $\bar{B}_{R}(x)$; we write $B_{R}$ and $\bar{B}_{R}$ for $B_{R}(0)$ and $\bar{B}_{R}(0)$, respectively; $\chi_{B}$ is the characteristic function of $B\subset\R^{d}$; the norm and the inner product on $L^{2}=L^{2}(\R^{d})$ are denoted by $\|\cdot\|$ and $\lan\cdot,\cdot\ran$; for two Banach spaces $X$ and $Y$, we denote by $\LLL(X,Y)$ the space of all bounded linear operators from $X$ to $Y$ and the norm on $\LLL(X,Y)$ is denoted by $\|\cdot\|=\|\cdot\|_{\LLL(X,Y)}$; if $X=Y$, we write $\LLL(X)=\LLL(X,Y)$; the set $\N$ stands for positive integers;  $L^{\infty}=L^{\infty}(\R^{d})$ with norm $\|\cdot\|_{\infty}$; $H^{2}=H^{2}(\R^{d})$; $\ell^{p}=\ell^{p}(\N)$.


\section{Basic Settings}\label{sec-settings}

This section serves as a preparation for the proof of main results. As in the assumptions of Theorem \ref{thm-main}, $\{v_{n}\}_{n\in\N}$ are real-valued and satisfy $\sup_{n\in\N}\|v_{n}\|_{\infty}<\infty$ and $\cup_{n\in\N}\supp(v_{n})\subset\Om$ for some bounded set $\Om\subset\R^{d}$. The sequence $\{x_{n}\}_{n\in\N}\subset\R^{d}$ satisfies the sparse condition
\begin{equation}\label{sparseness-2}
{\rm dist}\big(x_{n},\{x_{m}\}_{m\neq n}\big)\gtrsim n^{\ga},\quad n\in\N
\end{equation}
for some $\ga>\frac{2}{d-1}$. Then, the sparse potential
\begin{equation*}
V=\sum_{n=1}^{\infty}v_{n}(\cdot-x_{n})
\end{equation*}
is real-valued and bounded, and hence, Kato-Rellich theorem (see e.g. \cite[Theorem X.12]{RS75}) ensures that $H=H_{0}+V=-\De+V$ defines a lower bounded self-adjoint operator on $L^{2}$ with domain $H^{2}$.

Let $R>0$ be such that $\supp(v_{n})\subset \bar{B}_{R}$ for all $n\in\N$.
Also, for $n\in\N$, let 
\begin{equation*}
\Si_{n}=\supp(v_{n}(\cdot-x_{n})). 
\end{equation*}
Then, $\Si_{n}\subset\bar{B}_{R}(x_{n})$. We will use the following simple result.

\begin{lem}\label{lem-tech-1}
For any $m,n\in\N$, there holds
\begin{equation*}
\Si_{m}\times\Si_{n}\subset \bar{B}_{R}(x_{m})\times\bar{B}_{R}(x_{n})\subset\bar{B}_{\sqrt{2}R}(x_{m},x_{n})\subset\Big\{(x,y)\in\R^{d}\times\R^{d}\Big|\big||x-y|-|x_{m}-x_{n}|\big|\leq2R\Big\}.
\end{equation*}
\begin{proof}
The first two inclusions are trivial. If $(x,y)\in\bar{B}_{\sqrt{2}R}(x_{m},x_{n})$, then
\begin{equation*}
\begin{split}
|x-y|&\leq|x-x_{m}|+|x_{m}-x_{n}|+|x_{n}-y|\\
&\leq\sqrt{2}|(x,y)-(x_{m},x_{n})|+|x_{m}-x_{n}|\\
&\leq2R+|x_{m}-x_{n}|.
\end{split}
\end{equation*}
Similarly, $|x_{m}-x_{n}|\leq2R+|x-y|$. This proves the third one.
\end{proof}
\end{lem}

\begin{rem}\label{rem-tech-1}
Since, by Kuroda-Birman theorem (see e.g. \cite[Theorem XI.9]{RS79}), only the behavior of $V$ at infinity matters when we study the scattering theory of $H$, we may drop $v_{n}(\cdot-x_{n})$, $n=1,\dots,N-1$ for some $N$ large. For later use, this large $N$ is choosen to satisfy
\begin{itemize}
\item[\rm(i)] $\Si_{n}$, $n\geq N$ are pairwise disjoint;
\item[\rm(ii)] $(\Si_{m}\times\Si_{n})\cap\big\{(x,y)\in\R^{d}\times\R^{d}\big||x-y|\leq2R\big\}=\emptyset$ for any $m,n\geq N$ with $m\neq n$.
\end{itemize}
Condition $\rm(i)$ is easily satisfied by the sparseness condition \eqref{sparseness-2}, which together with Lemma \ref{lem-tech-1}, ensures $\rm(ii)$. Also, $\rm(ii)$ simply means that if $m\neq n$, then $\Si_{m}\times\Si_{n}$ is away from the diagonal $\big\{(x,y)\in\R^{d}\times\R^{d}\big|x=y\big\}$.
\end{rem}

From now on, we fix some large $N$ as in the above remark and write
\begin{equation*}
V=\sum_{n=N}^{\infty}v_{n}(\cdot-x_{n})\quad\text{and}\quad \Si=\bigcup_{n=N}^{\infty}\Si_{n}.
\end{equation*}

For $z\in\C\bs\R$, we set $R_{0}(z)=(H_{0}-z)^{-1}$ and $R(z)=(H-z)^{-1}$, the free resolvent and the perturbed resolvent. It is well-known (see e.g. \cite[page 288]{GS77}) that $R_{0}(z)$ is an integral operator with an explicit integral kernel given by
\begin{equation}\label{int-kernel-free}
k_{0,z}(x,y)=k_{0,z}(|x-y|)=c_{d}\bigg(\frac{\sqrt{z}}{|x-y|}\bigg)^{(d-2)/2}K_{(d-2)/2}(-i\sqrt{z}|x-y|),
\end{equation}
where $c_{d}$ is a constant depends only on $d$,  $\sqrt{z}$ satisfies $\Im\sqrt{z}>0$ and $K_{\bullet}$ is the Bessel potential. In particular, if $d=3$, then $k_{0,z}(x,y)=\frac{e^{i\sqrt{\la}|x-y|}}{4\pi|x-y|}$. Using Kato's smooth method to prove the existence and completeness of wave operators, the continuous extension of $R_{0}(z)$ and $R(z)$ to $\R_{+}$ plays a crucial role. For $R_{0}(z)$, we have the well-known limiting absorption principle, that is, for $z\in\C\bs\R$ with $\Re z>0$ and $\si>\frac{1}{2}$
\begin{equation*}
\|R_{0}(z)\|_{\LLL(L^{2}_{\si},L^{2}_{-\si})}\leq C_{\si}(\Re z)^{-1/2},
\end{equation*}
where $C_{\si}>0$ depends only on $\si$ and $L^{2}_{s}=\{\phi|\|(1+|\cdot|^{2})^{s/2}\phi\|<\infty\}$ is the weighted space. As for $R(z)$ with short-range $V$, i.e., $(1+|\cdot|)^{1+\ep}V\in L^{\infty}$ for some $\ep>0$, Agmon's fundamental work (see \cite{Ag75}) established the following bound
\begin{equation}\label{lap-perturbed}
\sup_{\Re z>\la,\Im z>0}\|R(z)\|_{\LLL(L^{2}_{\si},L^{2}_{-\si})}<\infty
\end{equation}
provided $\la>0$ and $\si>\frac{1}{2}$. Later, similar results were obtained for other classes of decaying potentials (see e.g. \cite{GS04,IS06,RS78,Ya00}). 

However, the potential $V$ discussed in the current paper may not decay fast enough at infinity or even not decay, and therefore,  \eqref{lap-perturbed} is not expected here. To overcome this difficulty, we modify both the free resolvent $R_{0}(z)$ and the perturbed resolvent $R(z)$ by defining the following operators: for $z\in\C\bs\R$, we set
\begin{equation*}
\begin{split}
F(z)&=|V|^{1/2}R_{0}(z)V^{1/2},\\
P(z)&=|V|^{1/2}R(z)V^{1/2},
\end{split}
\end{equation*}
where $V^{1/2}={\sgn}(V)|V|^{1/2}=|V|^{1/2}{\sgn}(V)$ so that $V=V^{1/2}|V|^{1/2}$. Since $V$ is bounded, both $F(z)$ and $P(z)$ are bounded on $L^{2}$. By means of the resolvent identity $R(z)-R_{0}(z)=-R(z)VR_{0}(z)$, we get
\begin{equation*}
|V|^{1/2}R(z)V^{1/2}-|V|^{1/2}R_{0}(z)V^{1/2}=-|V|^{1/2}R(z)V^{1/2}|V|^{1/2}R_{0}(z)V^{1/2},
\end{equation*}
that is,
\begin{equation*}
P(z)(1+F(z))=F(z).
\end{equation*}
Therefore, the limiting absorption principle, i.e. Theorem \ref{thm-main}$\rm(i)$, may be established if enough information about $F(z)$ can be acquired.

In the sequel, we fix $0<a<b<\infty$ and set
\begin{equation*}
\begin{split}
\SS&=\Big\{z=\la+i\ep\in\C_{+}\Big|\la\in[a,b]\,\,\text{and}\,\,\ep\in(0,1]\Big\},\\
\ol{\SS}&=\Big\{z=\la+i\ep\in\C_{+}\Big|\la\in[a,b]\,\,\text{and}\,\,\ep\in[0,1]\Big\}.
\end{split}
\end{equation*}
In the following two sections, we prove for $z\in\SS$ the uniform boundedness of $F(z)$ and of $(1+F(z))^{-1}$, which then allows us to establish the uniform boundedness of $P(z)$. 


\section{Analysis of $F(z)$}\label{sec-free}

Recall that for $z\in\SS$, $F(z)$ is defined by
\begin{equation*}
F(z)=|V|^{1/2}R_{0}(z)V^{1/2}.
\end{equation*}
Since $|V|^{1/2}$ and $V^{1/2}$ are integral operators with integral kernels $|V|^{1/2}(x)\de(x-y)$ and $V^{1/2}(x)\de(x-y)$, respectively, $F(z)$ is also an integral operator with integral kernel
\begin{equation}\label{integral-kernel}
\begin{split}
k_{F(z)}(x,y)&=\int_{\R^{d}}\int_{\R^{d}}|V|^{1/2}(x)\de(x-u)k_{0,z}(u,v)V^{1/2}(v)\de(v-y)dudv\\
&=|V|^{1/2}(x)k_{0,z}(x,y)V^{1/2}(y),
\end{split}
\end{equation}
where $k_{0,z}(x,y)$ is the integral kernel of $R_{0}(z)$ given in \eqref{int-kernel-free}. We recall the following property of $k_{0,z}(x,y)$ (see e.g. \cite[page 418]{IS06}), which comes from the standard estimates of Bessel potentials. 
\begin{lem}\label{estimate-of-kernel}
There's a constant $C=C(a,b)$ such that for any $z\in\ol{\SS}$
\begin{equation*}
\begin{split}
|k_{0,z}(x,y)|=|k_{0,z}(|x-y|)|\leq \left\{ \begin{aligned}
C\ln\bigg(\frac{2}{|x-y|}\bigg),&\quad\text{if}\,\,|x-y|\leq2R\,\,\text{and}\,\,d=2,\\
\frac{C}{|x-y|^{d-2}},&\quad\text{if}\,\,|x-y|\leq2R\,\,\text{and}\,\,d\geq3,\\
\frac{C}{|x-y|^{(d-1)/2}},&\quad\text{if}\,\,|x-y|>2R\,\,\text{and}\,\,d\geq2.
\end{aligned} \right.
\end{split}
\end{equation*}
\end{lem}

We now prove the uniform boundedness of $F(z)$ for $z\in\SS$.

\begin{thm}\label{lem-free-res}
$F(z)$ is uniformly bounded on $\SS$, that is,
\begin{equation*}
\sup_{z\in\SS}\|F(z)\|<\infty.
\end{equation*}
\end{thm}
\begin{proof}
Let $\phi,\psi\in L^{2}$, we estimate $\lan\phi,F(z)\psi\ran$. Using \eqref{integral-kernel}, we have
\begin{equation}\label{an-equality}
\begin{split}
\lan\phi,F(z)\psi\ran&=\iint_{\R^{d}\times\R^{d}}\ol{\phi(x)}|V|^{1/2}(x)k_{0,z}(x,y)V^{1/2}(y)\psi(y)dxdy\\
&=\sum_{n=N}^{\infty}\iint_{\R^{d}\times\R^{d}}\ol{\phi(x)}|v_{n}(x-x_{n})|^{1/2}k_{0,z}(x,y)v_{n}(y-x_{n})^{1/2}\psi(y)dxdy\\
&\quad+\sum_{m,n\geq N,m\neq n}\iint_{\R^{d}\times\R^{d}}\ol{\phi(x)}|v_{m}(x-x_{m})|^{1/2}k_{0,z}(x,y)v_{n}(y-x_{n})^{1/2}\psi(y)dxdy\\
&\quad=I+II,
\end{split}
\end{equation}
where $v_{n}(\cdot-x_{n})^{1/2}={\sgn}(v_{n}(\cdot-x_{n}))|v_{n}(\cdot-x_{n})|^{1/2}$. Note by Lemma \ref{lem-tech-1} and Remark \ref{rem-tech-1}, if $x,y\in\Si_{n}$ then $|x-y|\leq2R$, and if $x\in\Si_{m}$, $y\in\Si_{n}$ with $m\neq n$ then $|x-y|>2R$.

For the term $I$ in the last line of \eqref{an-equality}, we claim that
\begin{equation}\label{estimate-1}
|I|\lesssim\|\chi_{\Si}\phi\|\|\chi_{\Si}\psi\|.
\end{equation}
Considering Lemma \ref{estimate-of-kernel}, we distinguish between $d=2$ and $d\geq3$. We first deal with the case $d\geq3$. If $d\geq3$, we obtain from Lemma \ref{estimate-of-kernel} and the boundedness of $V$ that
\begin{equation*}
\begin{split}
|I|&\lesssim\sum_{n=N}^{\infty}\iint_{\Si_{n}\times\Si_{n}}\frac{|\phi(x)||\psi(y)|}{|x-y|^{d-2}}dxdy\\
&=\sum_{n=N}^{\infty}\iint_{\Si_{n}\times\Si_{n}}\frac{|\phi(x)|}{|x-y|^{(d-2)/2}}\frac{|\psi(y)|}{|x-y|^{(d-2)/2}}dxdy\\
&\leq\sum_{n=N}^{\infty}\bigg(\iint_{\Si_{n}\times\Si_{n}}\frac{|\phi(x)|^{2}}{|x-y|^{d-2}}dxdy\bigg)^{1/2}\bigg(\iint_{\Si_{n}\times\Si_{n}}\frac{|\psi(x)|^{2}}{|x-y|^{d-2}}dxdy\bigg)^{1/2}
\end{split}
\end{equation*}
Using the inequality $\chi_{\Si_{n}}(x)\chi_{\Si_{n}}(y)\leq\chi_{\Si_{n}-\Si_{n}}(x-y)$, we find
\begin{equation*}
\begin{split}
\iint_{\Si_{n}\times\Si_{n}}\frac{|\phi(x)|^{2}}{|x-y|}dxdy&=\iint_{\R^{d}\times\R^{d}}|\chi_{\Si_{n}}(x)\phi(x)|^{2}\frac{\chi_{\Si_{n}}(x)\chi_{\Si_{n}}(y)}{|x-y|^{d-2}}dxdy\\
&\leq\iint_{\R^{d}\times\R^{d}}|\chi_{\Si_{n}}(x)\phi(x)|^{2}\frac{\chi_{\Si_{n}-\Si_{n}}(x-y)}{|x-y|^{d-2}}dxdy\\
&=\|\chi_{\Si_{n}}\phi\|^{2}\int_{\Si_{n}-\Si_{n}}\frac{1}{|x|^{d-2}}dx.
\end{split}
\end{equation*}
Similarly, 
\begin{equation*}
\iint_{\Si_{n}\times\Si_{n}}\frac{|\psi(x)|^{2}}{|x-y|^{d-2}}dxdy\leq\|\chi_{\Si_{n}}\psi\|^{2}\int_{\Si_{n}-\Si_{n}}\frac{1}{|x|^{d-2}}dx.
\end{equation*}
It then follows that
\begin{equation*}
\begin{split}
|I|&\lesssim\sum_{n=N}^{\infty}\|\chi_{\Si_{n}}\phi\|\|\chi_{\Si_{n}}\psi\|\int_{\Si_{n}-\Si_{n}}\frac{1}{|x|^{d-2}}dx\\
&\leq\bigg(\int_{\bar{B}_{2R}}\frac{1}{|x|^{d-2}}dx\bigg)\sum_{n=N}^{\infty}\|\chi_{\Si_{n}}\phi\|\|\chi_{\Si_{n}}\psi\|\\
&\leq\bigg(\int_{\bar{B}_{2R}}\frac{1}{|x|^{d-2}}dx\bigg)\|\chi_{\Si}\phi\|\|\chi_{\Si}\psi\|,
\end{split}
\end{equation*}
where we used the fact that for any $n\geq N$, $\Si_{n}-\Si_{n}\subset\bar{B}_{2R}$, and so $\int_{\Si_{n}-\Si_{n}}\frac{1}{|x|^{d-2}}dx\leq\int_{\bar{B}_{2R}}\frac{1}{|x|^{d-2}}dx$, and H\"{o}lder's inequality for sequence in the last step. Since $\int_{\bar{B}_{2R}}\frac{1}{|x|^{d-2}}dx<\infty$, we obtain \eqref{estimate-1} in the case $d\geq3$.

The estimate \eqref{estimate-1} in the case $d=2$ can be treated similarly. In fact, Lemma \ref{estimate-of-kernel}, the boundedness of $V$ and arguments as in the case $d\geq3$ give
\begin{equation*}
\begin{split}
|I|&\lesssim\bigg(\int_{\bar{B}_{2R}}\bigg|\ln\frac{2}{|x|}\bigg|dx\bigg)\|\chi_{\Si}\phi\|\|\chi_{\Si}\psi\|\\
&=\bigg(2\pi\int_{0}^{2R}\bigg|\ln\frac{2}{r}\bigg|rdr\bigg)\|\chi_{\Si}\phi\|\|\chi_{\Si}\psi\|\\
&\lesssim\|\chi_{\Si}\phi\|\|\chi_{\Si}\psi\|,
\end{split}
\end{equation*}
since the integral $\int_{0}^{2R}\big|\ln\frac{2}{r}\big|rdr<\infty$. Thus, \eqref{estimate-1} also holds in the case $d=2$.

For the term $II$ in the last line of \eqref{an-equality}, we claim that
\begin{equation}\label{estimate-2}
|II|\lesssim\|\chi_{\Si}\phi\|\|\chi_{\Si}\psi\|.
\end{equation}
In fact, H\"{o}lder's inequality applied to integrals with respect to $x$ and $y$ successively gives
\begin{equation*}
\begin{split}
|II|&\lesssim\sum_{m,n\geq N,m\neq n}\bigg(\int_{\Si_{m}}|\phi(x)|^{2}dx\bigg)^{1/2}\bigg(\iint_{\Si_{m}\times\Si_{n}}\frac{1}{|x-y|^{d-1}}dxdy\bigg)^{1/2}\bigg(\int_{\Si_{n}}|\psi(y)|^{2}dy\bigg)^{1/2}\\
&\leq\bigg(\sum_{m,n\geq N,m\neq n}\iint_{\Si_{m}\times\Si_{n}}\frac{1}{|x-y|^{d-1}}dxdy\bigg)^{1/2}\bigg(\sum_{m,n\geq N,m\neq n}\int_{\Si_{m}}|\phi(x)|^{2}dx\int_{\Si_{n}}|\psi(y)|^{2}dy\bigg)^{1/2}\\
&\leq\bigg(\iint_{\cup_{m,n\geq N,m\neq n}\Si_{m}\times\Si_{n}}\frac{1}{|x-y|^{d-1}}dxdy\bigg)^{1/2}\|\chi_{\Si}\phi\|\|\chi_{\Si}\psi\|.
\end{split}
\end{equation*}
For the integral in the last line of the above estimate, symmetry gives
\begin{equation*}
\begin{split}
\iint_{\cup_{m,n\geq N,m\neq n}\Si_{m}\times\Si_{n}}\frac{1}{|x-y|^{d-1}}dxdy&\lesssim\iint_{\cup_{m,n\geq N,m<n}\Si_{m}\times\Si_{n}}\frac{1}{|x-y|^{d-1}}dxdy\\
&=\sum_{n\geq N}\iint_{\cup_{m:N\leq m<n}\Si_{m}\times\Si_{n}}\frac{1}{|x-y|^{d-1}}dxdy.
\end{split}
\end{equation*}
Since $|x_{m}-x_{n}|\gtrsim n^{\ga}$ for $m<n$ by the sparseness condition \eqref{sparseness-2}, we deduce from Remark \ref{rem-tech-1} that $|x-y|\gtrsim n^{\ga}$ for $(x,y)\in\Si_{m}\times\Si_{n}$ if $m<n$. Thus, for any $n\geq N$
\begin{equation*}
\iint_{\cup_{m:N\leq m<n}\Si_{m}\times\Si_{n}}\frac{1}{|x-y|^{d-1}}dxdy\lesssim n^{-(d-1)\ga}\bigg|\bigcup_{m:N\leq m<n}\Si_{m}\times\Si_{n}\bigg|\lesssim n^{-(d-1)\ga+1},
\end{equation*}
which implies that
\begin{equation*}
\iint_{\cup_{m,n\geq N,m\neq n}\Si_{m}\times\Si_{n}}\frac{1}{|x-y|^{2}}dxdy\lesssim\sum_{n\geq N}n^{-(d-1)\ga+1}.
\end{equation*}
Since $\ga>\frac{2}{d-1}$, i.e., $(d-1)\ga-1>1$, the above series converges. Hence, \eqref{estimate-2} follows.

Combining \eqref{an-equality}, \eqref{estimate-1} and \eqref{estimate-2}, we find for any $\phi,\psi\in L^{2}(\R^{3})$ the estimate
\begin{equation*}
|\lan\phi,F(z)\psi\ran|\lesssim\|\chi_{\Si}\phi\|\|\chi_{\Si}\psi\|,
\end{equation*}
which leads to the result since the above estimate is clearly independent of $z\in\SS$.
\end{proof}

Theorem \ref{lem-free-res} suggests us to define for $\la\in[a,b]$ the boundary operator 
\begin{equation}\label{def-boundary-op}
F(\la+i0)=\lim_{\ep\da0}F(\la+i\ep)\quad\text{in}\,\,\LLL(L^{2}).
\end{equation}
Note for $\la\in[a,b]$, the integral kernel $k_{F(\la+i\ep)}$ given in \eqref{integral-kernel} converges pointwise to 
\begin{equation*}
k_{F(\la+i0)}(x,y):=|V|^{1/2}(x)k_{0,\la}(x,y)V^{1/2}(y)
\end{equation*}
as $\ep\da0$. Moreover, since $k_{F(\la+i0)}$ clearly defines a bounded linear operator on $L^{2}$ by the proof of Theorem \ref{lem-free-res}, $F(\la+i0)$ must be the integral operator with integral kernel $k_{F(\la+i0)}$, that is,
\begin{equation}\label{boundary-op}
F(\la+i0)\phi(x)=\int_{\R^{3}}|V|^{1/2}(x)k_{0,\la}(x,y)V^{1/2}(y)\phi(y)dy,\quad\phi\in L^{2}.
\end{equation}
Then, considering the continuity of the integral kernels $k_{F(z)}$ in $z\in\ol{\SS}$, we obtain the continuity of $F(z)$ on $\ol{\SS}$. The above scenarios are summarized in

\begin{lem}\label{continuity-free}
The following statements hold.
\begin{itemize}
\item[\rm(i)] For each $\la\in[a,b]$, $F(\la+i\ep)$ is norm Cauchy as $\ep\da0$. In particular, $F(\la+i0)$, given in \eqref{def-boundary-op}, is well-defined.

\item[\rm(ii)] For each $\la\in[a,b]$, $F(\la+i0)$ is an integral operator with integral kernel $k_{F(\la+i0)}$, that is,  \eqref{boundary-op} is true.

\item[\rm(iii)] $F(z)$ is analytic in ${\rm int}(\SS)$ and continuous on $\ol{\SS}$. In particular, 
\begin{equation*}
\sup_{z\in\ol{\SS}}\|F(z)\|<\infty.
\end{equation*}

\end{itemize}
\end{lem}
\begin{proof}
Let $z_{1},z_{2}\in\SS$, the proof of Theorem \ref{lem-free-res} gives
\begin{equation*}
\begin{split}
\|F(z_{1})-F(z_{2})\|&\lesssim\int_{\bar{B}_{2R}}|k_{0,z_{1}}(|x|)-k_{0,z_{2}}(|x|)|dx\\
&\quad+\bigg(\iint_{\cup_{m,n\geq N,m\neq n}\Si_{m}\times\Si_{n}}|k_{0,z_{1}}(x,y)-k_{0,z_{2}}(x,y)|^{2}dxdy\bigg)^{1/2}.
\end{split}
\end{equation*}
Since, by the proof of Theorem 
\ref{lem-free-res}, the integrals $\int_{\bar{B}_{2R}}\frac{1}{|x-y|^{d-2}}dxdy$ (if $d\geq3$), $\int_{\bar{B}_{2R}}\big|\ln\frac{2}{|x|}\big|dx$ (if $d=2$) and $\iint_{\cup_{m,n\geq N,m\neq n}\Si_{m}\times\Si_{n}}\frac{1}{|x-y|^{d-1}}dxdy$ are finite , $\rm(i)$ and $\rm(ii)$ follow from Lebesgue's dominated convergence theorem. This also implies the continuity of $F(z)$ on $\ol{\SS}$. The analyticity of $F(z)$ in ${\rm int}(\SS)$, the interior of $\SS$, is a simple result of the analycity of $R_{0}(z)$ in ${\rm int}(\SS)$. Thus, $\rm(iii)$ is true.
\end{proof}

\section{Analysis of $P(z)$}\label{sec-perturbed}

Recall that for $z\in\SS$,
\begin{equation*}
P(z)=|V|^{1/2}R(z)V^{1/2}
\end{equation*}
and it satisfies
\begin{equation}\label{res-id}
P(z)(1+F(z))=F(z),
\end{equation}
where $F(z)=|V|^{1/2}R_{0}(z)V^{1/2}$. The aim of this section is to establish the uniform boundedness of $P(z)$ on $\SS$. The main result is following
\begin{thm}\label{thm-perturb-res}
$P(z)$ is uniformly bounded on $\SS$, that is, 
\begin{equation*}
\sup_{z\in\SS}\|P(z)\|<\infty.
\end{equation*}
\end{thm}

Considering the equation \eqref{res-id} and Theorem \ref{lem-free-res}, to prove Theorem \ref{res-id}, it suffices to prove the invertibility of $1+F(z)$ for each $z\in\SS$ and the uniform boundedness of their inverses on $\SS$, which are accomplished through several lemmas.

We first establishes the invertibility of $1+F(z)$ for $z\in\SS$.

\begin{lem}\label{invertible-away-from-real}
For each $z\in\SS$, the operator $1+F(z)$ is boundedly invertible. Moreover, $(1+F(z))^{-1}$ is continuous on $\SS$.
\end{lem}
\begin{proof}
Fix any $z\in\SS$. We first show that $-1$ is neither an eigenvalue nor in the residue spectrum of $F(z)$. Clearly, it suffices to show that $1+F(z)$ is one-to-one and has dense range.  To do so, let $\phi\in L^{2}$ be such that $(1+F(z))\phi=0$, then ${\sgn}(V)(1+F(z))\phi=0$, that is,
\begin{equation*}
{\sgn}(V)\phi+V^{1/2}R_{0}(z)V^{1/2}\phi=0.
\end{equation*}
Writing $V^{1/2}R_{0}(z)V^{1/2}=\Re(V^{1/2}R_{0}(z)V^{1/2})+i\Im(V^{1/2}R_{0}(z)V^{1/2})$, we have
\begin{equation*}
\lan\phi,{\sgn}(V)\phi\ran+\lan\phi,\Re(V^{1/2}R_{0}(z)V^{1/2})\phi\ran+i\lan\phi,\Im(V^{1/2}R_{0}(z)V^{1/2})\phi\ran=0,
\end{equation*}
which implies that $\lan\phi,\Im(V^{1/2}R_{0}(z)V^{1/2})\phi\ran=0$. Since
\begin{equation*}
\Im(V^{1/2}R_{0}(z)V^{1/2})=V^{1/2}\Im R_{0}(z)V^{1/2}=\Im zV^{1/2}R_{0}(\ol{z})R_{0}(z)V^{1/2},
\end{equation*}
we conclude that $\Im z\|R_{0}(z)V^{1/2}\phi\|^{2}=0$, which implies $R_{0}(z)V^{1/2}\phi=0$, and so, $V^{1/2}\phi=0$. Therefore, $F(z)\phi=0$ and
\begin{equation*}
\phi=(1+F(z))\phi=0.
\end{equation*}
This shows that $1+F(z)$ is one-to-one. A similar argument shows that $(1+F(z))^{*}=1+F(z)^{*}$ is also one-to-one. From the fact that $\ol{{\range}(1+F(z))}\oplus\ker(1+F(z)^{*})=L^{2}$, we conclude that $1+F(z)$ has dense range. Hence, $1+F(z)$ is one-to-one and has dense range. 

As a result of the fact that $-1$ is neither an eigenvalue nor in the residue spectrum of $F(z)$, we obtain that if $1+F(z)$ is not boundedly invertible, then $(1+F(z))^{-1}$ is densely defined and unbounded. 

Now, we show that $1+F(z)$ is boundedly invertible. For contradiction, we assume that $1+F(z)$ is not boundedly invertible, that is, $-1\in\si(F(z))$. Then, the above analysis says that there exsits $\{\phi_{n}\}_{n\in\N}\subset L^{2}$ such that $\|\phi_{n}\|=1$ for all $n$ and $\|(1+F(z))\phi_{n}\|\ra0$ as $n\ra\infty$. Define 
\begin{equation*}
\psi_{n}=R_{0}(z)V^{1/2}\phi_{n}.
\end{equation*}
Since $V$ is bounded, $\psi_{n}\in H^{2}$. Moreover, there's some $C>0$ such that $\inf_{n\in\N}\|\psi_{n}\|\geq C$. In fact, if this is not true, then we can find some subsequence $\{\phi_{n_{k}}\}_{k\in\N}$ such that $\|\psi_{n_{k}}\|=\|R_{0}(z)V^{1/2}\phi_{n_{k}}\|\ra0$ as $k\ra\infty$. It then follows that
\begin{equation*}
\begin{split}
1=\|\phi_{n_{k}}\|&\leq\|(1+F(z))\phi_{n_{k}}\|+\|F(z)\phi_{n_{k}}\|\\
&\leq\|(1+F(z))\phi_{n_{k}}\|+\||V|^{1/2}\|_{\infty}\|R_{0}(z)V^{1/2}\phi_{n_{k}}\|\\
&\ra0\quad\text{as}\,\,k\ra\infty,
\end{split}
\end{equation*}
which leads to the contradiction. 

Also, setting $\vp_{n}=(1+F(z))\phi_{n}$, i.e., $\phi_{n}=\vp_{n}-F(z)\phi_{n}$, we have
\begin{equation*}
(H_{0}-z)\psi_{n}=V^{1/2}\phi_{n}=V^{1/2}\vp_{n}-V^{1/2}F(z)\phi_{n}=V^{1/2}\vp_{n}-V\psi_{n},
\end{equation*}
that is, $(H_{0}+V-z)\psi_{n}=V^{1/2}\vp_{n}$, or
\begin{equation*}
(H_{0}+V-z)\frac{\psi_{n}}{\|\psi_{n}\|}=\frac{V^{1/2}\vp_{n}}{\|\psi_{n}\|}.
\end{equation*}
Since $\inf_{n\in\N}\|\psi_{n}\|\geq C$, we have $\|\frac{V^{1/2}\vp_{n}}{\|\psi_{n}\|}\|\leq\frac{\|V^{1/2}\|_{\infty}\|\vp_{n}\|}{C}\ra0$ as $n\ra\infty$. The bounded invertibility of $H_{0}+V-z$ then ensures that
\begin{equation*}
\frac{\psi_{n}}{\|\psi_{n}\|}=(H_{0}+V-z)^{-1}\frac{V^{1/2}\vp_{n}}{\|\psi_{n}\|}\ra0\quad\text{as}\,\,n\ra\infty,
\end{equation*}
which leads to a contradiction. Consequently, $-1\in\rho(F(z))$, that is, $1+F(z)$ is boundedly invertible.

For the continuity of $(1+F(z))^{-1}$ on $\SS$, we fix any $z_{0}\in\SS$, and then for any $z\in\SS$ we have the formal expansion 
\begin{equation*}
\begin{split}
(1+F(z))^{-1}&=[1+F(z_{0})+F(z)-F(z_{0})]^{-1}\\
&=[(1+F(z_{0}))(1+(1+F(z_{0}))^{-1}(F(z)-F(z_{0})))]^{-1}\\
&=[1-(1+F(z_{0}))^{-1}(F(z_{0})-F(z))]^{-1}(1+F(z_{0}))^{-1}\\
&=\bigg(1+\sum_{n=1}^{\infty}[(1+F(z_{0}))^{-1}(F(z_{0})-F(z))]^{n}\bigg)(1+F(z_{0}))^{-1}.
\end{split}
\end{equation*}
The above series converges if $\|(1+F(z_{0}))^{-1}(F(z_{0})-F(z))\|<1$, which, by Lemma \ref{continuity-free}, is true for all $z$ close to $z_{0}$. Thus, for any $z\in\SS$ close to $z_{0}$, we deduce
\begin{equation*}
\begin{split}
&\|(1+F(z))^{-1}-(1+F(z_{0}))^{-1}\|\\
&\quad\quad\leq\bigg\|\bigg(\sum_{n=1}^{\infty}[(1+F(z_{0}))^{-1}(F(z_{0})-F(z))]^{n}\bigg)(1+F(z_{0}))^{-1}\bigg\|\\
&\quad\quad\leq\frac{\|(1+F(z_{0}))^{-1}(F(z_{0})-F(z))\|}{1-\|(1+F(z_{0}))^{-1}(F(z_{0})-F(z))\|}\|(1+F(z_{0}))^{-1}\|\\
&\quad\quad\ra0\quad\text{as}\quad z\ra z_{0}.
\end{split}
\end{equation*}
This establishes the continuity of $(1+F(z))^{-1}$ at $z_{0}$, and thus, $(1+F(z))^{-1}$ is continuous on $\SS$.
\end{proof}

Next, we prove the invertibility of $1+F(\la+i0)$ for $\la\in[a,b]$. To do so, we need the following result about approximate inverse (see e.g. \cite{EZ}).

\begin{prop}\label{approximate-inverse}
Let $X,Y$ be Banach spaces and $A\in\LLL(X,Y)$. Suppose that there exist $B_{1},B_{2}\in\LLL(Y,X)$, $R_{1}\in\LLL(Y)$ with $\|R_{1}\|_{\LLL(Y)}<1$ and $R_{2}\in\LLL(X)$ with $\|R_{2}\|_{\LLL(X)}<1$ such that
\begin{equation*}
\begin{split}
AB_{1}&=I_{Y}+R_{1}\,\,\text{on}\,\,Y\\
B_{2}A&=I_{X}+R_{2}\,\,\text{on}\,\,X,
\end{split}
\end{equation*}
where $I_{X}$ and $I_{Y}$ are identity operators on $X$ and $Y$, respectively. Then, $A$ is boundedly invertible with $A^{-1}:Y\ra X$ given by
\begin{equation*}
A^{-1}=B_{1}(I_{Y}+R_{1})^{-1}=(I_{X}+R_{2})^{-1}B_{2}.
\end{equation*}
\end{prop}

Now, we prove

\begin{lem}\label{invertible-on-the-real}
For each $\la\in[a,b]$, the operator $1+F(\la+i0)$ is boundedly invertible.
\end{lem}
\begin{proof}
Fix any $\la\in[a,b]$. We first claim that there exists $C>0$ such that 
\begin{equation}\label{bound-below}
\inf_{\ep\in(0,1]}\|1+F(\la+i\ep)\|\geq C.
\end{equation}
In fact, if this is not true, then, by Lemma \ref{invertible-away-from-real}, there exsits $\{\ep_{n}\}_{n\in\N}$ such that $\ep_{n}\ra0$ as $n\ra\infty$ and 
\begin{equation*}
\|1+F(\la+i\ep_{n})\|\ra0\quad\text{as}\,\,n\ra\infty.
\end{equation*}
This means that $\{F(\la+i\ep_{n})\}_{n\in\N}$ converges in norm to the operator $-I$, where $I$ is the identity opertor on $L^{2}$. Since $\{F(\la+i\ep_{n})\}_{n\in\N}$ also converges in norm to the operator $F(\la+i0)$ by Lemma \ref{continuity-free}$\rm(i)$, we have $F(\la+i0)=-I$. But, clearly, this is not the case since $F(\la+i0)$ is the integral operator with the integral kernel $|V|^{1/2}(x)k_{0,\la}(x,y)V^{1/2}(y)$,
and thus, if $\phi=0$ on $\Si$, then $F(\la+i0)\phi=0$. Hence, \eqref{bound-below} is true. By \eqref{bound-below} and Lemma \ref{invertible-away-from-real}, we find
\begin{equation}\label{bound-above}
\sup_{\ep\in(0,1]}\|(1+F(\la+i\ep))^{-1}\|\leq\frac{1}{C}.
\end{equation}

Next, we apply Proposition \ref{approximate-inverse} to prove the bounded invertibility of $1+F(\la+i0)$. Note
\begin{equation*}
\begin{split}
&(1+F(\la+i0))(1+F(\la+i\ep))^{-1}\\
&\quad\quad=[1+F(\la+i\ep)+F(\la+i0)-F(\la+i\ep)](1+F(\la+i\ep))^{-1}\\
&\quad\quad=1+[F(\la+i0)-F(\la+i\ep)](1+F(\la+i\ep))^{-1}.
\end{split}
\end{equation*}
Simiarly, 
\begin{equation*}
(1+F(\la+i\ep))^{-1}(1+F(\la+i0))=1+(1+F(\la+i\ep))^{-1}[F(\la+i0)-F(\la+i\ep)]
\end{equation*}
Thus, if we fix some $\ep\in(0,1]$ such that $\|F(\la+i0)-F(\la+i\ep)\|<C$ (this can be done by Lemma \ref{continuity-free}), \eqref{bound-above} implies
\begin{equation*}
\|(F(\la+i0)-F(\la+i\ep))(1+F(\la+i\ep))^{-1}\|<1
\end{equation*}
and
\begin{equation*}
\|(1+F(\la+i\ep))^{-1}(F(\la+i0)-F(\la+i\ep))\|<1.
\end{equation*}
By Proposition \ref{approximate-inverse}, we conclude that $1+F(\la+i0)$ is boundedly invertible.
\end{proof}

In Lemma \ref{invertible-away-from-real} and Lemma \ref{invertible-on-the-real}, we proved that $1+F(z)$ is boundedly invertible for each $z\in\ol{\SS}$. In the next result, we prove the uniform boundedness of their inverses.

\begin{lem}\label{uniform-bounded-inverse}
$(1+F(z))^{-1}$ is continuous on $\ol{\SS}$. In particular, there holds 
\begin{equation*}
\sup_{z\in\ol{\SS}}\|(1+F(z))^{-1}\|<\infty.
\end{equation*}
\end{lem}
\begin{proof}
Considering the compactness of $\ol{\SS}$, it suffices to show the continuity of $(1+F(z))^{-1}$ on $\ol{\SS}$. By Lemma \ref{invertible-away-from-real}, we only need to show the continuity of $(1+F(z))^{-1}$ at $z=\la+i0$ for each $\la\in[a,b]$. The latter actualy follows from the continuity of $F(z)$ on $\ol{\SS}$ by Lemma \ref{continuity-free} as in the proof of Lemma \ref{invertible-away-from-real}. Indeed, fixing any $\la\in[a,b]$, for any $z\in\ol{\SS}$, we have the formal expansion 
\begin{equation*}
(1+F(z))^{-1}=\bigg(1+\sum_{n=1}^{\infty}[(1+F(\la+i0))^{-1}(F(\la+i0)-F(z))]^{n}\bigg)(1+F(\la+i0))^{-1}.
\end{equation*}
The series converges if $\|(1+F(\la+i0))^{-1}(F(\la+i0)-F(z))\|<1$, which, by Lemma \ref{continuity-free}, is true for all $z$ close to $\la$. This completes the proof.
\end{proof}

Finally, we prove Theorem \ref{res-id}.

\begin{proof}[Proof of Theorem \ref{res-id}]
The result follows from \eqref{res-id}, Lemma \ref{continuity-free}, Lemma \ref{invertible-away-from-real}, Lemma \ref{invertible-on-the-real} and Lemma \ref{uniform-bounded-inverse}. In fact, $P(z)$ can be continuously extended to $\ol{\SS}$. In particular, there holds
\begin{equation*}
\sup_{z\in\ol{\SS}}\|P(z)\|<\infty.
\end{equation*}
\end{proof}


\section{Proof of Main Results and Remarks}\label{sec-proof-main-results}

In this section, we first prove Theorem \ref{thm-main}, Theorem \ref{thm-main-1} and Theorem \ref{cor-main}. Then, we make some remarks.

\begin{proof}[Proof of Theorem \ref{thm-main}]
The first part of Theorem \ref{thm-main} is a direct consequence of Theorem \ref{lem-free-res} and Theorem \ref{thm-perturb-res}.

For the existence and completeness of local wave operators, we invoke Proposition \ref{app-local-wave-op}. To do so, writing $H-H_{0}=|V|^{1/2}V^{1/2}$, since clearly $|V|^{1/2}$ is $H$-bounded and $V^{1/2}$ is $H_{0}$-bounded, we only need to show that for any compact interval $I\subset(0,\infty)$, $|V|^{1/2}$ is $H$-smooth on $I$ and $V^{1/2}$ is $H_{0}$-smooth on $I$. 

Let $I\subset(0,\infty)$ be a compact interval. For the $H$-smoothness of $|V|^{1/2}$ on $I$, Proposition \ref{app-local-smoothness} says that it suffices to show that 
\begin{equation}\label{smoothness}
\sup_{\la\in I,\ep\in(0,1]}\||V|^{1/2}R(\la+i\ep)|V|^{1/2}\|<\infty.
\end{equation}
Now, by $|V|^{1/2}=V^{1/2}{\sgn}(V)$, we have
\begin{equation*}
\||V|^{1/2}R(\la+i\ep)|V|^{1/2}\|\leq\||V|^{1/2}R(\la+i\ep)V^{1/2}\|\|{\sgn}(V)\|=\||V|^{1/2}R(\la+i\ep)V^{1/2}\|.
\end{equation*}
Theorem \ref{thm-perturb-res} then implies \eqref{smoothness}. Similarly, the $H_{0}$-smoothness of $V^{1/2}$ on $I$ follows from Theorem \ref{lem-free-res}. This completes the proof.
\end{proof}

Next, we prove Theorem \ref{thm-main-1}.
\begin{proof}[Proof of Theorem \ref{thm-main-1}]
We invoke Proposition \ref{app-wave-op}. To do so, let $\{I_{n}\}_{n\in\N}$ be a family of compact intervals in $(0,\infty)$ such that $(0,\infty)=\cup_{n=1}^{\infty}I_{n}$. By the proof of Theorem \ref{thm-main}, we know that $|V|^{1/2}$ is $H$-smooth on each $I_{n}$ and $V^{1/2}$ is $H_{0}$-smooth on each $I_{n}$. To apply Proposition \ref{app-wave-op}, we only need to show that $\si(H)\bs(0,\infty)$ has Lebesgue measure zero.

Clearly, by Weyl's theorem (see e.g. \cite[Theorem XIII.14]{RS78}), if $\{v_{n}\}_{n\in\N}$ satisfies condition $\rm(i)$ in the statment of Theorem \ref{thm-main-1}, then $\si(H)\bs(0,\infty)$ has Lebesgue measure zero. If $\{v_{n}\}_{n\in\N}$ satisfies condition $\rm(ii)$ in the statment of Theorem \ref{thm-main-1}, we need to use Klaus's theorem (see e.g. \cite{CFKS87,HK00,Kla83}) about the structure of the essential spectrum of $H$, that is, 
\begin{equation*}
\si_{\rm ess}(H)=\EE\cup\R_{+}.
\end{equation*}
Since $\EE$ has Lebesgue measure zero by assumption, we find
\begin{equation*}
|\si(H)\bs(0,\infty)|=|\si_{\rm disc}(H)|+|\si_{\rm ess}(H)\cap(-\infty,0]|=0.
\end{equation*}
Thus, the wave operators exist and are complete.

The conclusion $\si_{ac}(H)=\R_{+}$ follows from general results of the existence and completeness of wave operators (see e.g. \cite[Corollary 3.6.6]{DK05}). This completes the proof.
\end{proof}

We now prove Theorem \ref{cor-main}.

\begin{proof}[Proof of Theorem \ref{cor-main}]
Under the assumption of Corollary \ref{cor-main}, the set $\EE$ is given by
\begin{equation*}
\EE=\si_{\rm disc}(H_{0}+v),
\end{equation*}
and thus, it has zero Lebesgue measure.
\end{proof}

Finally, we make remarks about the results of the current paper.

\begin{rem}\label{rem-final}
\begin{itemize}
\item[\rm(i)] We construct an example showing that the set $\EE$ in the statement of Theorem \ref{thm-main-1} can have positive Lebesgue measure. Let $v$ be a nonpositive, bounded and compactly supported function and suppose $E(1)<0$ is an eigenvalue of $H_{0}+v$. Then, by analytic perturbation theory (see e.g. \cite{Ka76,RS78}), there's a small $\beta_{0}>0$ such that for any $\beta\in(1-\beta_{0},1+\beta_{0})$ there's a unique point $E(\beta)\in\si(H_{0}+\beta v)$ near $E(1)$. Moreover, $E(\beta)$ is analytic. Also, since $v\leq0$, $E(\beta)$ is strictly decreasing by Feynman-Hellmann theorem (see e.g. \cite[Theorem 4.1.29]{St01}). We define the set $\{\beta_{n}\}_{n\in\N}$ by
\begin{equation*}
\{\beta_{n}\}_{n\in\N}:=E^{-1}\big((E(1+\beta_{0}),E(1-\beta_{0}))\cap\Q)\big).
\end{equation*}
It is countable and for any $E\in(E(1+\beta_{0}),E(1-\beta_{0}))$, we can find a subsequence $\{\beta_{n_{k}}\}_{k\in\N}$ such that $E(\beta_{n_{k}})\ra E$ as $k\ra\infty$. Now, consider the operator
\begin{equation*}
H=H_{0}+\sum_{n=1}^{\infty}v_{n}(\cdot-x_{n})
\end{equation*}
with $v_{n}=\beta_{n}v$ for $n\in\N$. Since $\{\beta_{n}\}_{n\in\N}\subset(1-\beta_{0},1+\beta_{0})$, $\{v_{n}\}_{n\in\N}$ are uniformly bounded and uniformly compactly supported. Moreover, above analysis says that $\EE$ contains the set $(E(1+\beta_{0}),E(1-\beta_{0}))$, and therefore has positive Lebesgue measure.

\item[\rm(ii)] Note that $I$ and $II$ in \eqref{an-equality} correspond to the diagonal and off-diagonal part of $F(z)$. The proof of Theorem \ref{lem-free-res} clearly implies that the off-diagonal part of $F(z)$ is Hilbert-Schmidt. Although this compactness does not play a role in our proof, it does play a key role in the work of Jak\v{s}i\'{c} and Poulin (see \cite{JP09}) for the discrete model. Moreover, if we write
$F(z)=F_{diag}(z)+F_{offdiag}(z)$. Then, while $\Im F_{diag}(z)$ is strictly positive for the discrete model (this directly leads to the bounded invertibility of $1+F_{diag}(z)$), it is not true for the continuum model discussed here. In fact, assuming $v_{n}$, $n\in\N$ are nonnegative, then
\begin{equation*}
F_{diag}(z)=\sum_{n=N}^{\infty}v_{n}(\cdot-x_{n})^{1/2}R_{0}(z)v_{n}(\cdot-x_{n})^{1/2},
\end{equation*}
which leads to
\begin{equation*}
\Im F_{diag}(z)=\Im z\sum_{n=N}^{\infty}v_{n}(\cdot-x_{n})^{1/2}R_{0}(\bar{z})R_{0}(z)v_{n}(\cdot-x_{n})^{1/2}
\end{equation*}
by resolvent identity. Thus, for any $z\in\SS$, $\Im F_{diag}(z)\geq0$. But, note that each $v_{n}(\cdot-x_{n})^{1/2}R_{0}(\bar{z})R_{0}(z)v_{n}(\cdot-x_{n})^{1/2}$ is compact, and so, its spectrum contains $0$. Hence, $0\in\si(\Im F_{diag}(z))$, and $\Im F_{diag}(z)$ is not strictly positive.

\item[\rm(iii)] We remark that the proof of Theorem \ref{res-id}, that is,
\begin{equation*}
\sup_{z\in\SS}\|P(z)\|<\infty,
\end{equation*}
can be greatly simplified if $V$ is decaying. In fact, if $V$ is decaying, that is, $\|v_{n}\|_{\infty}\ra0$ as $n\ra\infty$, then Theorem \ref{lem-free-res} actually reads
\begin{equation*}
\sup_{z\in\SS}\|F(z)\|\lesssim\sup_{n\geq N}\|v_{n}\|_{\infty}
\end{equation*}
Therefore, we can fix some large $N$ so that 
$\sup_{z\in\SS}\|F(z)\|\leq\frac{1}{2}$. By stability of bounded invertibility (see e.g. \cite{La02}), for each $z\in\SS$, $1+F(z)$ is boundedly invertible with 
\begin{equation*}
\|(1+F(z))^{-1}\|\leq\frac{1}{1-\|F(z)\|}\leq2.
\end{equation*}
It then follows that for each $z\in\SS$, $\|P(z)\|=\|F(z)(1+F(z))^{-1}\|\leq1$.
\end{itemize}
\end{rem}


\appendix

\section{Smooth Perturbation}\label{smooth-perturb}

We collect some results of Kato's smooth method used in Section \ref{sec-proof-main-results}. The following material is taken from \cite[Section XIII.7]{RS78}. We also refer to \cite[Chapter 4]{Ya92}. 

Let $H_{0}$ and $H$ be self-adjoint operators on some separable Hilbert space $\HH$. Let $A$ be a closed operator on $\HH$. We recall that $A$ is called \textit{$H$-smooth} if for each $\phi\in\HH$ and each $\ep\neq0$, $(H-\la-i\ep)^{-1}\phi\in\DD(A)$ for almost all $\la\in\R$ and 
\begin{equation*}
\sup_{\|\phi\|=1,\ep>0}\int_{\R}\big(\|A(H-\la-i\ep)^{-1}\phi\|^{2}+\|A(H-\la+i\ep)^{-1}\phi\|^{2}\big)d\la<\infty.
\end{equation*}
We also recall that $A$ is called \textit{$H$-smooth on $\Om$}, a Borel set, if $A\chi_{\Om}(H)$ is $H$-smooth. The following result gives a sufficient condition for $A$ being $H$-smooth on $\Om$.

\begin{prop}\label{app-local-smoothness}
Let $\Om\subset\R$ and suppose $\DD(H)\subset\DD(A)$. Then, $A$ is $H$-smooth on $\ol{\Om}$ provided
\begin{equation*}
\sup_{\la\in\Om,\ep\in(0,1)}\|A(H-\la-i\ep)A^{*}\|<\infty.
\end{equation*}
\end{prop}

Regarding the existence and completeness of local wave operators and wave operators, we have the following two results.

\begin{prop}\label{app-local-wave-op}
Suppose that there are closed operators $A$, $B$ on $\HH$ such that
\begin{equation*}
H-H_{0}=A^{*}B
\end{equation*}
in the sense of quadratic forms. Suppose further that $A$ is $H$-bounded and $H$-smooth on some bounded open interval $I\subset\R$ and that $B$ is $H_{0}$-bounded and $H_{0}$-smooth on $I$. Then, the local wave operators $\Om_{\pm}(H,H_{0};I)=s\mbox{-}\lim_{t\ra\pm\infty}e^{iHt}e^{-iH_{0}t}P_{ac}(H_{0})\chi_{I}(H_{0})$ exist and are complete.
\end{prop}

\begin{prop}\label{app-wave-op}
Suppose that there are closed operators $A$, $B$ on $\HH$ such that $A$ is $H$-bounded and $B$ is $H_{0}$-bounded, and that
\begin{equation*}
H-H_{0}=A^{*}B
\end{equation*}
in the sense of quadratic forms. Let $\{I_{i}\}_{i=1}^{\infty}$ be a family of bounded open intervals and let $S=\cup_{i=1}^{\infty}I_{i}$. Suppose that
\begin{itemize}
\item[\rm(i)] $A$ is $H$-smooth on each $I_{i}$ and $B$ is $H_{0}$-smooth on each $I_{i}$.
\item[\rm(ii)] Both $\si(H)\bs S$ and $\si(H_{0})\bs S$ have Lebesgue measure zero.
\end{itemize}
Then, the wave operators $\Om_{\pm}(H,H_{0})=s\mbox{-}\lim_{t\ra\pm\infty}e^{iHt}e^{-iH_{0}t}P_{ac}(H_{0})$ exist and are complete.
\end{prop}



\begin{thebibliography}{99}
\bibitem{Ag75} S. Agmon, Spectral properties of Schr\"{o}dinger operators and scattering theory. \textit{Ann. Scuola Norm. Sup. Pisa Cl. Sci. (4)} 2 (1975), no. 2, 151–218.





\bibitem{CFKS87} H. L. Cycon, R. G. Froese, W. Kirsch and B. Simon, \textit{Schr\"{o}dinger operators with application to quantum mechanics and global geometry.} Texts and Monographs in Physics. Springer Study Edition. Springer-Verlag, Berlin, 1987.

\bibitem{De08} S. A. Denisov, Wave propagation through sparse potential barriers. \textit{Comm. Pure Appl. Math.} 61 (2008), no. 2, 156-185.
 
\bibitem{DK05} M. Demuth and M. Krishna, \textit{Determining spectra in quantum theory.} Progress in Mathematical Physics, 44. Birkh\"{a}user Boston, Inc., Boston, MA, 2005.








\bibitem{EZ} L. C. Evans and M. Zworski, \textit{Lectures on semiclassical analysis.} Available from the website: http://math.berkeley.edu/$\sim$evans/semiclassical.pdf.



\bibitem{GS77} I. M. Gel'fand and G. E. Shilov, \textit{Generalized functions. Vol. 1. Properties and operations.} Translated from the Russian by Eugene Saletan. Academic Press [Harcourt Brace Jovanovich, Publishers], New York-London, 1964 [1977].

\bibitem{GS04} M. Goldberg and W. Schlag, A limiting absorption principle for the three-dimensional Schr\"{o}dinger equation with $L^{p}$ potentials. \textit{Int. Math. Res. Not.} 2004, no. 75, 4049–4071.

\bibitem{HK00} D. Hundertmark and W. Kirsch, Spectral theory of sparse potentials. Stochastic processes, physics and geometry: new interplays, I (Leipzig, 1999), 213-238, CMS Conf. Proc., 28, Amer. Math. Soc., Providence, RI, 2000.





\bibitem{IS06} A. Ionescu and W. Schlag, Agmon-Kato-Kuroda theorems for a large class of perturbations. \textit{Duke Math. J.} 131 (2006), no. 3, 397–440.

\bibitem{JL03} V. Jak\v{s}i\'{c} and Y. Last, Scattering from subspace potentials for Schr\"{o}dinger operators on graphs. \textit{Markov Process. Related Fields.} 9 (2003), no. 4, 661–674.

\bibitem{JP09} V. Jak\v{s}i\'{c} and Ph. Poulin, Scattering from sparse potentials: a deterministic approach. \textit{Analysis and mathematical physics, 205–210, Trends Math., Birkh\"{a}user, Basel,} 2009.

\bibitem{Ka65} T. Kato, Wave operators and similarity for some non-selfadjoint operators. \textit{Math. Ann.} 162 1965/1966 258–279.

\bibitem{Ka76} T. Kato, \textit{Perturbation theory for linear operators.} Second edition. Grundlehren der Mathematischen Wissenschaften, Band 132. Springer-Verlag, Berlin-New York, 1976.


\bibitem{Ki02} W. Kirsch, Scattering theory for sparse random potentials. \textit{Random Oper. Stochastic Equations} 10 (2002), no. 4, 329–334.





\bibitem{Kla83} M. Klaus, On $-\frac{d^{2}}{dx^{2}}+V$ where $V$ has infinitely many ``bumps''. \textit{Ann. Inst. H. Poincar\'{e} Sect. A (N.S.)} 38 (1983), no. 1, 7-13.























\bibitem{Kri93} M. Krishna, Absolutely continuous spectrum for sparse potentials. \textit{Proc. Indian Acad. Sci. Math. Sci.} 103 (1993), no. 3, 333–339.


\bibitem{Kru04} D. Krutikov, Schr\"{o}dinger operators with random sparse potentials. Existence of wave operators. \textit{Lett. Math. Phys.} 67 (2004), no. 2, 133–139.

\bibitem{KLS98} A. Kiselev, Y. Last and B. Simon, Modified Pr\"{u}fer and EFGP transforms and the spectral analysis of one-dimensional Schr\"{o}dinger operators. \textit{Comm. Math. Phys.} 194 (1998), no. 1, 1-45.

\bibitem{La02} P. Lax, \textit{Functional analysis.} Pure and Applied Mathematics (New York). Wiley-Interscience [John Wiley \& Sons], New York, 2002.











\bibitem{Mo98} S. Molchanov, Multiscattering on sparse bumps. Advances in differential equations and mathematical physics (Atlanta, GA, 1997), 157–181, \textit{Contemp. Math.} 217, Amer. Math. Soc., Providence, RI, 1998.

\bibitem{MV99} S. Molchanov and B. Vainberg, Scattering on the system of the sparse bumps: multidimensional case. \textit{Appl. Anal.} 71 (1999), no. 1-4, 167-185.

\bibitem{MV00} S. Molchanov and B. Vainberg, Spectrum of multidimensional Schr\"{o}dinger operators with sparse potentials. \textit{Analytical and computational methods in scattering and applied mathematics (Newark, DE, 1998)}, 231–254, Chapman $\&$ Hall/CRC Res. Notes Math., 417, Chapman $\&$ Hall/CRC, Boca Raton, FL, 2000.







\bibitem{Pe79} D. B. Pearson, Singular continuous measures in scattering theory. \textit{Comm. Math. Phys.} 60 (1978), no. 1, 13-36. 

\bibitem{Po08} Ph. Poulin, Scattering from sparse potentials on graphs. \textit{Zh. Mat. Fiz. Anal. Geom.} 4 (2008), no. 1, 151-170, 204.





\bibitem{RS75} M. Reed and B. Simon, \textit{Methods of modern mathematical physics. II. Fourier analysis, self-adjointness.} Academic Press [Harcourt Brace Jovanovich, Publishers], New York-London, 1975.

\bibitem{RS79} M. Reed and B. Simon, \textit{Methods of modern mathematical physics. III. Scattering theory.} Academic Press [Harcourt Brace Jovanovich, Publishers], New York-London, 1979.

\bibitem{RS78} M. Reed and B. Simon, \textit{Methods of modern mathematical physics. IV. Analysis of operators.} Academic Press [Harcourt Brace Jovanovich, Publishers], New York-London, 1978.





\bibitem{St01} P. Stollmann, \textit{Caught by disorder. Bound states in random media.} Progress in Mathematical Physics, 20. Birkh\"{a}user Boston, Inc., Boston, MA, 2001.





\bibitem{Ya92} D. Yafaev, \textit{Mathematical scattering theory. General theory.} Translated from the Russian by J. R. Schulenberger. Translations of Mathematical Monographs, 105. American Mathematical Society, Providence, RI, 1992.

\bibitem{Ya00} D. Yafaev, \textit{Scattering theory: some old and new problems.} Lecture Notes in Mathematics, 1735. Springer-Verlag, Berlin, 2000.

\end{thebibliography}
\end{document}